\documentclass{amsart}

\usepackage{amsmath}
\usepackage{amsfonts}
\usepackage{amssymb}
\usepackage{amsthm}
\usepackage{mathrsfs}
\usepackage{yfonts}
\usepackage{enumitem}
\usepackage{tikz}
\usepackage{tikz-cd}
\usepackage{amsaddr}

\newtheorem{theorem}{Theorem}[section]
\newtheorem*{claim}{Claim}
\newtheorem{lemma}[theorem]{Lemma}
\newtheorem{proposition}[theorem]{Proposition}
\newtheorem{corollary}[theorem]{Corollary}

\newtheorem{question}[theorem]{Question}
\newtheorem{definition}[theorem]{Definition}

\newtheorem*{acknowledgement}{Acknowledgements}

\newcommand{\forces}[1]{\Vdash``#1"}
\newcommand{\baire}{\omega^\omega}

\newcommand{\bbb}{\textgoth{b}}
\newcommand{\ddd}{\textgoth{d}}

\newcommand{\mmm}{\textgoth{m}}

\newcommand{\continuum}{\mathfrak{c}}

\newcommand{\Borel}{\mathrm{Borel}}

\newcommand{\Meager}{\mathcal{M}}
\newcommand{\Null}{\mathcal{N}}

\newcommand{\Baire}{\omega^{\omega}}
\newcommand{\Cantor}{2^{\omega}}

\newcommand{\TT}{\mathbb{T}}
\newcommand{\LL}{\mathbb{L}}
\newcommand{\MM}{\mathbb{M}}
\newcommand{\ideal}[1]{\mathcal{#1}}
\newcommand{\FM}{\mathbb{FM}}

\newcommand{\RandomGraph}{\mathcal{R}}
\newcommand{\Solecki}{\mathcal{S}}
\newcommand{\ED}{\mathcal{ED}}
\newcommand{\FF}{\mathcal{F}in\times\mathcal{F}in}
\newcommand{\nwd}{nwd}
\newcommand{\FIN}{\mathcal{F}in}
\newcommand{\Z}{\mathcal{Z}}
\newcommand{\SUM}{\mathcal{S}um}
\newcommand{\Fin}{\mathcal{F}in}

\title{On ideals related to Laver and Miller trees}
\author{Aleksander Cieślak}
\address{Faculty of Pure and Applied Mathematics, Wrocław University of Science and Technology, Wybrzeże Stanisława Wyspiańskiego 27, 50-370 Wrocław, Poland}

\author{Arturo Martínez-Celis}
\address{Instytut Matematyczny, Uniwersytet Wrocławski, pl. Grunwaldzki 2, 50-384 Wrocław, Poland}

\email{aleksander.cieslak@pwr.edu.pl, arturo.martinez-celis@math.uni.wroc.pl}

\begin{document}
\maketitle

\begin{abstract}
In this work we consider the ideals $m^0(\mathcal{I})$ and $\ell^0(\mathcal{I})$, ideals generated by the $\mathcal{I}$-positive Miller trees and $\mathcal{I}$-positive Laver trees, respectively. We investigate in which cases these ideals have cofinality larger than $\mathfrak{c}$ and we calculate some cardinal invariants closely related to these ideals.
\end{abstract}

\section{Introduction}
The \emph{Marczewski ideal} $s^{0}$, introduced in \cite{Marcz}, is the collection of all $X\subseteq2^{\omega}$ with the property that every perfect tree on $2^{<\omega}$ has a perfect subtree that shares no branches with $X$. In that work, the author noticed that $s^{0}$ does not have a Borel basis, i.e. there are sets in $s^{0}$ that cannot be covered by a Borel set in $s^{0}$. Generalizing this, Fremlin showed that $\mathrm{cof}(s^{0})>\continuum$ holds in ZFC (see \cite{MiJuSh}). By replacing the perfect trees on $2^{<\omega}$ with Laver or Miller trees on $\omega^{<\omega}$ we obtain the \emph{Laver ideal} $\ell^{0}$ and the \emph{Miller ideal} $m^{0}$, respectively. These ideals and their cofinalities have been of interest in recent years. For example, in \cite{BreKhWoh} authors considered properties of tree types which implies that the cofinality of the tree ideal is greater than continuum, in \cite{Repicky} the author deepened this analysis for Laver trees and, in \cite{ShSpinCofin} the authors consider the problem of consistently distinguishing cofinalities of different types of trees. This paper is concerned with the study of cofinalities of the ideals related to the following variants of Laver and Miller trees: If $\mathcal{I}$ is an ideal on $\omega$ then
\begin{itemize}
    \item a tree $T\subseteq \omega^{<\omega}$ is an \emph{$\mathcal{I}$-Miller tree} if for every $\sigma\in T$ there is $\tau\in T$, $\sigma\subseteq\tau$ such that the set $succ_{T}(\tau)$ is not in $\mathcal{I}$. The collection of all $\mathcal{I}$-Miller trees will be denoted by $\MM_{\mathcal{I}}$,
 
    \item a tree $T\subseteq \omega^{<\omega}$ is an \emph{$\mathcal{I}$-Laver tree} if for every $\sigma\in T$ extending $stem(T)$, the set $succ_{T}(\tau)$ is not in $\mathcal{I}$. The collection of all $\mathcal{I}$-Laver trees will be denoted by $\LL_{\mathcal{I}}$. 
\end{itemize}

In particular, if $\mathcal{I}$ is the ideal of finite subsets of $\omega$, then $\LL_{\mathcal{I}}$ is just the classical Laver forcing, and $\MM_{\mathcal{I}}$ the classical Miller forcing.

To analyze these cofinalities, we will follow the work of Brendle, Khomskii, and Wohofsky. In \cite{BreKhWoh}, the authors found two general properties for a tree type that guarantees that $\mathrm{cof}(t^{0})>\continuum$. They also noticed that the tree type of \emph{full-spliting Miller trees} (\cite{RosNew}, \cite{KhomLag}), defined as
\begin{center}
    $\FM:=\{T\subseteq\omega^{<\omega}: \forall \sigma\in T \exists\tau\in T$ $\sigma\subseteq\tau$ $\forall i\in\omega$ $\tau ^{\frown} i\in T\}$ 
\end{center}
seems to be a misfit in both approaches, leaving $\mathrm{cof}(fm^{0})>\continuum$ as an open problem. Here, $fm^{0}$ is the tree ideal related to $\FM$. The results of \cite{BreKhWoh} are formulated under the following framework
\begin{definition}
    A collection of pruned trees $\TT$ on $\omega^{<\omega}$ is a \emph{combinatorial tree forcing} if for all $T\in\TT$:
\begin{enumerate}
    \item the restriction $T|_{\sigma}$ is also in $\TT$ for all $\sigma\in T$,
    \item for all $T \in \mathbb{T}$ there is a continuous function $f: [T] \rightarrow 2^\omega$ such that, for all $x \in \Cantor, f^{-1}(\{x\})$ is the set of branches of a tree in $\mathbb{T}$.
\end{enumerate}
    
    If additionally, the following condition is satisfied, then we will say that $\TT$ is an \emph{homogeneous combinatorial tree forcing.}
    
\begin{enumerate}
      \item[(3)] there is a homeomorphism $h:\omega^{\omega}\rightarrow[T]$ such that every set $X\subseteq \omega^{\omega}$ contains the body of a $\TT$-tree if and only if $f[[T]]$ does. \label{homogeneousproperty}
\end{enumerate}
\end{definition}

In the literature, the first property is usually referred as the arboreal property and both $\MM_\ideal{I}$ and $\LL_\ideal{I}$ satisfy it. The second property is a stronger version of being $\mathfrak{c}^+$-c.c. and one can easily verify that $\MM_\ideal{I}$ satisfies it, regardless of $\mathcal{I}$. Notice that  $\LL_\ideal{I}$ may fail to satisfy this property as $\LL_\ideal{I}$ could be a c.c.c. forcing notion (for example when $\ideal{I}$ is a maximal ideal). In general, $\LL_\ideal{I}$ satisfies the second property when $\ideal{I}$ is \emph{nowhere maximal}, i.e. when there is no $X \notin \ideal{I}$ such that the \emph{restricted ideal} $\ideal{I} |_X = \{ A \subseteq X : A \in \ideal{I} \}$ is not maximal.  The third condition is called \emph{the homogeneity property} (it is stated in an equivalent way in \cite{BreKhWoh}) and it is well known that both Miller and Laver trees satisfy it (see for example \cite{Zapl1}). Using the same argument, one can show that $\mathcal{I}$-Miller and $\mathcal{I}$-Laver trees satisfy this condition whenever $\mathcal{I}$ is an homogeneous ideal (see Proposition \ref{homog}). If the third condition is satisfied only for a dense set of $T \in \mathbb{T}$ we will say that $\TT$ is \emph{weakly homogeneous}. Sometimes the ideals that we are going to consider are not homogeneous, but the conclusions are still valid.

If $\TT$ is a tree type and $T, S\in \TT$ we write $S\leq T$ if $S\subseteq T$. Now, with every tree type $\TT$ there are two associated notions of smallness.
\begin{definition}
The \emph{tree ideal associated to} $\TT$, denoted by $t^{0}$, is the collection of all $X\subseteq \omega^{\omega}$ such that for any $T\in\TT$ there is $S\leq T$, $S\in\TT$ such that $[S]\cap X=\emptyset$. The ideal $t^{0}_{\mathcal{B}}$ is the Borel part of $t^0$, that is the $\sigma$-ideal generated by $t^{0}\cap\mathrm{Borel}(\baire)$.
\end{definition}
In general, $t^{0}$ may not be a $\sigma$-ideal, although it is going to be for all the tree types considered in this manuscript (a standard fusion argument should be enough to show that both $m^0$ and $\ell^{0}$ are $\sigma$-ideals, see Proposition \ref{aresigmaideals}).

One of the results of \cite{BreKhWoh} is the following:

\begin{theorem}\label{BreKhoWoh}
    Let $\TT$ be a combinatorial tree forcing. If $\TT$ satisfies one of the following two properties then $\mathrm{cof}(t^{0})>\continuum$ holds.
    \begin{itemize}
        \item \emph{Constant or 1-1 property}:
        
	       For every $T\in\TT$ and every continuous $f:[T]\rightarrow 2^{\omega}$ there is a tree $S\leq T, S\in\TT$ such that the restriction $f|_{[S]}$ is either constant or 1-1.
        \item \emph{Incompatibility shrinking property}:
        
            For every $T\in\TT$ and a collection $\{S_{\alpha}:\alpha<\continuum\}\subseteq \TT$ of trees which are incompatible with $T$, there is a $S\in\TT$, $S\subseteq T$ such that $[S]\cap[S_{\alpha}]=\emptyset$ for each $\alpha<\continuum$.
    \end{itemize}
    Additionally, if $\TT$ is a homogeneous combinatorial tree forcing, then the cofinality (as a cardinal) of $\mathrm{cof}(t^{0})$ is greater than $\continuum$.
\end{theorem}

The "Constant or 1-1 property" has been studied in the context of the minimality of forcing extensions. For example, in \cite{Sacks} G.E.Sacks proved the minimality of forcing extensions for Sacks trees and A. Miller proved that $\MM$ has Constant or 1-1 property (see \cite{MillerSuperperfect}). For Laver trees this property has been proved by C. Gray, in \cite{BreKhWoh}.  
In a more general setting, Sabok and Zapletal \cite{SabZapl} proved that if a forcing notion can be presented as $\Borel / \mathcal{J}$ for some $\sigma$-ideal $\mathcal{J}$ generated by closed sets, then the Constant or 1-1 property is equivalent to the minimality of a forcing extensions. The forcing $\FM$ falls into this class of forcing notions and adds Cohen reals (\cite{RosNew}), therefore $\FM$ does not have the constant or 1-1 property. In Theorem \ref{pplusconstant1} we prove that, for $P^+$-ideals, $\LL_{\mathcal{I}}$ does have the 'constant of 1-1 property'.

The Incompatibility shrinking property was isolated as a property in \cite{BreKhWoh} but it has been used in the past. In \cite{Gold} and \cite{MiJuSh} the authors implicitly used it to estimate the additivity of $s^{0}$, $m^{0}$ and $\ell^{0}$ to then consistently distinguish the additivities and covering numbers of the tree ideal. In \cite{Spinas} the author explicitly shown that the Incompatibility shrinking property for Miller forcing holds under $\ddd=\continuum$ also in the context of cardinal invariants of the tree ideals. In this paper, we express this property as a cardinal invariant: The \emph{incompatibility shrinking number} $\mathrm{is}(\TT)$ is the minimal amount of sets from $t^{0}_{\mathcal{B}}$ which intersects every $[T]$ for $T\in\TT$ and the \emph{hereditary shrinking number} $\mathrm{is}_h(\TT)$ is the smallest cardinal of the form $\mathrm{is}(\TT | T)$ for $T \in \TT$, where $\TT | T = \{ S \in \TT : S \leq T \}$. Note that, if $S,T \in \TT$ are such that $S \leq T$, then $\mathrm{is}(\TT | S) \leq \mathrm{is}(\TT | T)$, and therefore $\mathrm{is}_h(\TT) \leq \mathrm{is}(\TT)$ and, if $\TT$ is weakly homogeneous, these cardinals are equal. Observe that $\TT$ the Incompatibility shrinking property is equivalent to the assumption that $\mathrm{is}_h(\TT)=\continuum$. We prove that $min\{\bbb,\mathrm{add}^{*}(\mathcal{I})\}$ is a lower bound of $\mathrm{is}(\LL_{\mathcal{I}})$ and $\mathrm{is}(\MM_{\mathcal{I}})$ and $min\{\bbb,\mathrm{cov}(m^{0}_{\mathcal{B}}(\mathcal{I}))\}$ is a lower bound for $\mathrm{is}(\MM_{\mathcal{I}})$. On the other hand, it follows directly from the definition of $\mathrm{is}(\TT)$, that it falls in between $\mathrm{add}(t^{0}_{B})$ and $\mathrm{cov}(t^{0}_{B})$. However, for several combinatorial tree types, including $\FM$, the additivity of its Borel part is equal to $\omega_{1}$ in ZFC (see \cite{RosNew}). We will calculate the additivity and covering numbers of $m^{0}_{\mathcal{B}}(\mathcal{I})$ and $\ell^{0}_{\mathcal{B}}(\mathcal{I})$ for various Borel ideals. In particular, we show that for certain ideals, including the eventually different ideal, the ideal $conv$ and the Solecki ideal, the additivity of the Borel part is equal $\omega_{1}$ in ZFC. We also prove that for analytic P-ideals the additivity is greater than $\mathrm{add}(\Null)$.

The paper is organized as follows. In section 2 we consider when $\LL_{\mathcal{I}}$ or $\MM_{\mathcal{I}}$ have the constant or 1-1 property. Then, in section 3, we consider the incompatibility shrinking numbers. Finally, in section 4, we calculate the additivity and covering numbers of the Borel parts of the tree ideals of $\MM_{\mathcal{I}}$ and $\LL_{\mathcal{I}}$ for various Borel ideals. 

Our notation is mostly standard and follows \cite{Jech}. A tree is a subset of $\omega^{<\omega}$ closed under taking initial segments. Let $T$ be a tree and $\sigma\in T$, then $[T]=\{x\in \omega^{\omega}: \forall n\in\omega$ $x|_{n}\in T\}$ is the set of branches of $T$ i.e. the body of $T$, $T|_{\sigma}=: \{\tau\in T:\sigma\subseteq\tau \lor \tau\subseteq\sigma\}$ is the restriction of $T$ to $\sigma$, $succ_{T}(\sigma)=:\{n\in \omega:\sigma^{\frown}n \in T\}$ is the set of immediate successors of $\sigma$ in $T$ and $split(T)=:\{\sigma\in T:|succ_{T}(\sigma)|>1\}$ is the set of splitnodes of $T$. Also, $split_{n}(T)=\{\sigma\in split(T): |split(T)\cap \{\sigma|_{i}:i<|\sigma|\}|=n\}$. Clearly then $split(T)=\bigcup_{n}split_{n}(T)$. We will write $S\leq_{n} T$ when $S\leq T$ and $split_{n}(S)=split_{n}(T)$. If $succ_{T}(\sigma)=\emptyset$ then we call $\sigma$ a terminal node of $T$ and denote these as $term(T)$. Let $stem(T)$ be the shortest splitnode of $T$. Sometimes, when several trees with the same stem are considered, we will simply write $stem$.

An ideal $\mathcal{I}$ is a collection of subsets closed under finite unions and subsets. Sometimes, to avoid confusion, we will write $dom(\mathcal{I})$ for the set on which $\mathcal{I}$ lives. Subsets of $\omega$ which are not in $\mathcal{I}$ are referred to be $\mathcal{I}$-positive or simply positive if the ideal is known from context. An ideal $\mathcal{I}$ is tall if every infinite $A\subseteq\omega$ contains an infinite $B\in\mathcal{I}$. An ideal $\mathcal{I}$ is $P^+$-ideal if, for every almost decreasing sequence of $\mathcal{I}$-positive sets, there is a $\mathcal{I}$-positive pseudointersection i.e. an $A\in[\omega]^{\omega}\setminus\mathcal{I}$ almost contained in every element of the sequence.

We now list all the Borel ideals which are of particular importance for us in this paper. Let $\Fin$ be the ideal of finite subsets of $\omega$. All ideals we are going to consider in this paper contain $\Fin$. Let $\FF$ be ideal on $\omega\times\omega$ consisting of these sets $A$ such that for almost all $n\in\omega$ the set $\{m: (n,m)\in A\}$ is finite. Let $\ED$ be an ideal on $\omega\times\omega$ consisting of these sets $A$ such that for some $N$ for all $n>N$ the set $\{m: (n,m)\in A\}$ has at most $N$ elements. Let $\ED_{fin}$ be the ideal on $\{(a,b)\in\omega\times\omega:b\leq a\}$ consisting of these sets $A$ such that for some $N$ for all $n>N$ the set $\{m: (n,m)\in A\}$ has at most $N$ elements. Let $\nwd$ be the ideal on $2^{<\omega}$ consisting of these sets $A$ such that for any $\sigma\in2^{<\omega}$ there is $\tau\in2^{<\omega}$, $\sigma\subseteq\tau$ with $\tau\notin A$. Clearly, any $A\in nwd$ can be covered by the set of form of a tree $T\subseteq2^{<\omega}$ with $[T]$ being nowhere dense set in $2^{\omega}$. Let $\Solecki$ be the Solecki ideal, an ideal on $\Omega=:\{C\in CLOPEN(2^{\omega}):\lambda(C)=1/2\}$ generated by sets of the form $\{C\in\Omega: C\cap F\neq\emptyset\}$ where $F\in[2^{\omega}]^{<\omega}$. The random graph ideal $\RandomGraph$ is the ideal generated by cliques and anticliques in some fixed instance of the random graph. The convergent ideal $conv$ is defined as the ideal on $\mathbb{Q}\cap[0,1]$ generated by convergent sequences. The summable ideal $\SUM$ consists of these $A\subseteq\omega$ such that $\SUM_{n\in A}1/n<\infty$ and the density $\Z$ of these $A\subseteq\omega$ such that $lim_{n\to\infty}\frac{|A\cap n|}{n}<\infty$. For more information about these, the reader may consult \cite{Hrus}.

Given two ideals $\mathcal{I}_{0}$ and $\mathcal{I}_{1}$, then $\mathcal{I}_{0}$ is Katetov reducible to $\mathcal{I}_{1}$, $\mathcal{I}_{0}\leq_{K}\mathcal{I}_{1}$ for short, if there is a function $\pi:\omega\rightarrow\omega$ such that $\pi^{-1}[A]\in\mathcal{I}_{1}$ for any $A\in\mathcal{I}_{0}$. In such a case $\mathrm{cov}^*(\mathcal{I}_{0})\leq \mathrm{cov}^*(\mathcal{I}_{1})$ (see \cite{Hrus}). The ideal $\mathcal{I}$ is homogeneous if for each $\mathcal{I}$-positive set $X$ there is bijection $\pi:\omega\rightarrow X$ such that $\pi^{-1}[Y]$ is $\mathcal{I}$-positive iff $Y\subseteq X$ is. The ideal is \emph{weakly homogeneous} if this property is only satisfied for a dense set of positive sets. The ideals that we are going to consider in this work are not necessarily homogeneous (or even weakly homogeneous). The following diagram shows the complete picture of the mentioned ideals in the Katetov order (see \cite{HrushKatetov}). The arrows point toward the ideal which is Katetov stronger.

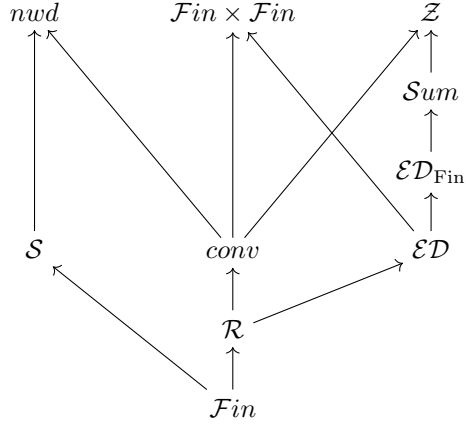
\begin{center}
\begin{tikzpicture}[]
  \matrix[matrix of math nodes,column sep={75pt,between origins},row
    sep={30pt,between origins},nodes={asymmetrical rectangle}] 
  {
    |[name=nwd]| \nwd &|[name=finfin]| \FF &|[name=Z]| \Z \\
    &&|[name=sum]| \SUM \\
    &&|[name=edfin]| \ED_{\mathrm{Fin}} \\
    |[name=sol]| \Solecki &|[name=conv]| conv &|[name=ed]| \ED  \\
    &|[name=rand]| \RandomGraph \\
    &|[name=fin]| \FIN \\
  };
    \draw[->] 
            (conv) edge (finfin)
            (ed) edge (finfin)
            (sum) edge (Z)
            (conv) edge (nwd)
            (conv) edge (Z)
            (rand) edge (ed)
            (ed) edge (edfin)
            (edfin) edge (sum)
            (rand) edge (conv)
            (fin) edge (rand)
            (fin) edge (sol)
            (sol) edge (nwd);
            
\end{tikzpicture}
\end{center}

A $\sigma$-ideal on a Polish space $X$ is an ideal $X$ closed under countable unions. As usual, $\mathcal{M}$ stands for the $\sigma$-ideal of meager sets in $X$ and $\mathcal{N}$ for the $\sigma$-ideal of sets of Lebesgue measure zero. If $\mathcal{J}$ is a $\sigma$-ideal on $X$, then the standard cardinal characteristics are
\begin{center}
    $\mathrm{add}(\mathcal{J})=\min\{|\mathcal{F}|:\mathcal{F}\subseteq\mathcal{J}$ and $\bigcup\mathcal{F}\notin\mathcal{J}\}$,\\
    $\mathrm{cov}(\mathcal{J})=\min\{|\mathcal{F}|:\mathcal{F}\subseteq\mathcal{J}$ and $\bigcup\mathcal{F}=X\}$,\\
    $\mathrm{non}(\mathcal{J})=\min\{|Y|:Y\subseteq X$ and $X\notin\mathcal{J}\}$,\\
    $\mathrm{cof}(\mathcal{J})=\min\{|\mathcal{F}|:\mathcal{F}\subseteq\mathcal{J}$ and $\forall Y\in\mathcal{J}\exists Y'\in\mathcal{F}$ $Y\subseteq Y'\}$.
\end{center}

For $f,g\in\omega^{\omega}$ we write $f\leq^{*} g$ if $g$ dominates $f$ i.e. $f(n)\leq g(n)$ for almost all $n\in\omega$. A family $\mathcal{F}\subseteq\omega^{\omega}$ is dominating if for every $g\in\omega^{\omega}$ there is $f\in\mathcal{F}$ which dominates $g$. A family $\mathcal{F}\subseteq\omega^{\omega}$ is unbounded if there is no $g\in\omega^{\omega}$ which dominates every $f\in\mathcal{F}$. The cardinal invariant $\bbb$ is defined as the minimal cardinality of an unbounded family and $\ddd$ as the minimal cardinality of a dominating family.

For $\Meager$ and $\Null$, the ZFC-provable inequalities between these invariants are summarised in the following diagram, widely known as the Cicho\'{n}'s diagram (the interested reader may consult \cite{BJ} or \cite{Blasschar}):

\begin{center}
\begin{tikzpicture}[]
  \matrix[matrix of math nodes,column sep={55pt,between origins},row
    sep={35pt,between origins},nodes={asymmetrical rectangle}] (s)
  {
    &|[name=covn]| \mathrm{cov}(\mathcal{N}) &|[name=nonm]| \mathrm{non}(\mathcal{M}) &|[name=cofm]| \mathrm{cof}(\mathcal{M}) &|[name=cofn]| \mathrm{cof}(\mathcal{N}) &|[name=continuum]| \continuum \\
    &&|[name=b]| \bbb &|[name=d]| \ddd \\
    |[name=omegaone]| \aleph_{1} &|[name=addn]| \mathrm{add}(\mathcal{N}) &|[name=addm]| \mathrm{add}(\mathcal{M}) &|[name=covm]| \mathrm{cov}(\mathcal{M}) &|[name=nonn]| \mathrm{non}(\mathcal{N}) \\
  };
    \draw[->] (omegaone) edge (addn)
            (addn) edge (addm)
            (addm) edge (covm)
            (covm) edge (nonn)
            (addm) edge (b)
            (b) edge (d)
            (covm) edge (d)
            (covn) edge (nonm)
            (addn) edge (covn)
            (b) edge (nonm)
            (d) edge (cofm)
            (nonm) edge (cofm)
            (cofm) edge (cofn)
            (nonn) edge (cofn)
            (cofn) edge (continuum);
\end{tikzpicture}
\end{center}
The arrows point towards the larger cardinal invariant.

\section{Constant or 1-1 property}
We begin this section with a couple of basic, already mentioned facts about the $\mathcal{I}$-Laver and $\mathcal{I}$-Miller trees. Let us start with the following proposition

\begin{proposition}\label{homog}
    Both $\MM_{\mathcal{I}}$ and $\LL_{\mathcal{I}}$ are (weakly) homogeneous whenever $\mathcal{I}$ is (weakly) homogeneous. 
\end{proposition}
\begin{proof}
    We verify this only for Laver trees and for the homogeneous property as the proof for the other cases is similar. Let then $T\in\LL_{\mathcal{I}}$ and for each $\sigma\in split(T)$ fix a function $\pi_{\sigma}:\omega\rightarrow succ_{T}(\sigma)$ such that $\pi_{\sigma}^{-1}[X]$ is $\mathcal{I}$-positive iff $X\subseteq succ_{T}(\sigma)$ is. Define inductively a homeomorphism $h:\omega^{\omega}\rightarrow [T]$ such that $h(x)=stem(T) ^{\frown} \pi_{\sigma_{0}}(x(0))^{\frown} \pi_{\sigma_{1}}(x(1))^{\frown}...$ where $\sigma_{n}=h(x)|_{|stem(T)|+n}$. It is easy to see that $h$ satisfies the required properties.
\end{proof}

The next proposition is easy to verify. 

\begin{proposition}\label{aresigmaideals}
    Both $\ell^{0}(\mathcal{I})$ and $m^{0}(\mathcal{I})$ are $\sigma$-ideals.
\end{proposition}
\begin{proof}
    Let $t^{0}_{\mathcal{B}}$ be any of these two ideals. If $\{X_{n}:n\in\omega\}\subseteq t^{0}$ and $T\in \TT$ we construct a fusion sequence $\{T_{n}:n\in\omega\}$ such that $[T_{n}]\cap X_{n}=\emptyset$. To construct $T_{n+1}$ out of $T_{n}$ in the case of Laver we use pure decision property while for Miller this can be done directly. Then, if $S$ is a fusion of $T_{n}$'s we clearly have $[S]\cap\bigcup_{n}X_{n}=\emptyset$.
\end{proof}

The Borel parts $\ell^{0}_{\mathcal{B}}(\mathcal{I})$ and $m^{0}_{\mathcal{B}}(\mathcal{I})$ are also $\sigma$-ideals by Theorems \ref{SabokZaplDich} and \ref{MillerDich} respectively. The crucial tool for investigating the constant or 1-1 property will be the possibility of reading reals out of the generic real in a continuous way.

\begin{definition}
A tree type $\TT$ has the \emph{continuous reading of names} if for each $T\in\TT$ and each $\TT$-name $\dot{x} \in 2^{\omega}$, there is a $S\leq T$, $S\in\TT$ and a continuous function $f:[S]\rightarrow 2^{\omega}$ such that $S \forces{f(\dot{r}_\mathrm{gen}) = \dot{x}}$.
\end{definition}

The continuous reading of names for $\MM_{\mathcal{I}}$ follows from a result of Zapletal \cite{Zapl1} saying that if a forcing notion $\TT$ can be expressed as $\Borel\setminus\mathcal{J}$ where $\sigma$-ideal $\mathcal{J}$ is generated by closed sets (and $\MM_{\mathcal{I}}$ can, by Theorem \ref{SabokZaplDich}) then it has the continuous reading of names. However, as $\LL_{\mathcal{I}}$ always adds a dominating real, the ideal $\ell^{0}_{\mathcal{B}}(\mathcal{I})$ is not $\Sigma_{1}^{1}$-on-$\Pi_{1}^{1}$ (see \cite{Zapl1}) and thus cannot be $\sigma$-generated by closed sets. We will show the continuous reading of names for $\LL_{\mathcal{I}}$ directly. To do so, first we show that $\LL_{\mathcal{I}}$ has the pure decision property.

\begin{proposition}[Pure decision property]
    If $T\in\LL_{\mathcal{I}}$ is such that $T\forces{\dot{x}\in 2}$ then there is $S\leq_{0}T$ which decides the value of $\dot{x}$ i.e. either $S\forces{\dot{x}=0}$ or $S\forces{\dot{x}=1}$.
\end{proposition}
\begin{proof}
    Assume towards contradiction that there is no such $S\leq_{0}T$ and let $T_{0}=T$. Inductively we construct a sequence $\{T_{n}:n\in\omega\}\subseteq \LL_{\mathcal{I}}$ such that for every $n\in\omega$ we have $T_{n+1}\leq_{n}T_{n}$ and for every $\sigma\in split_{n}(T)$ there is no $S\leq T_{n}|{\sigma}$ deciding the value of $\dot{x}$. Assume that $T_{n}$ hes been constructed. For each $\sigma\in split_{n}(T)$ the set $A_{\sigma}=:\{i\in succ_{T_{n}}(\sigma): \exists S\leq_{0}T_{n}|_{\sigma^{\frown}i} \text{ deciding }\dot{x}\}$ is in $\mathcal{I}$. We then construct $T_{n+1}$ by letting $succ_{T_{n+1}}(\sigma)=succ_{T_{n}}(\sigma)\setminus A_{\sigma}$ for $\sigma\in split_{n}(T)$. Finally, if $S$ is the fusion of $T_{n}$'s then there is no $R\leq S$ deciding $\dot{x}$, a contradiction.
\end{proof}

Now, having the pure decision property, the continuous reading of names follows easily.

\begin{proposition}
    $\LL_{\mathcal{I}}$ has the continuous reading of names.
\end{proposition}
\begin{proof}
    Let $T\forces{\dot{x}\in 2^{\omega}}$. Inductively we construct a sequence $\{T_{n}:n\in\omega\}\subseteq \LL_{\mathcal{I}}$ such that $T_{0}=T$, for every $n\in\omega$ $T_{n+1}\leq_{n}T_{n}$ and for every $\sigma\in split_{n}(T)$, $T_{n}|{\sigma}\forces{\dot{x}|_{n}=x_{\sigma}}$. Then, if $S$ is the fusion of $T_{n}$'s we define $f:[S]\rightarrow 2^{\omega}$ as $f(y)=\bigcup x_{y|_{n}}$. Clearly $f$ is continuous and maps generic real into $\dot{x}$.
\end{proof}

Let us recall the following theorem of Sabok and Zapletal which completely characterizes the constant or 1-1 property for $\mathcal{I}$-Miller trees.

\begin{theorem}[\cite{SabZapl}]\label{SabZapl}
Forcing $\MM_{\mathcal{I}}$ adds Cohen reals if and only if $\nwd\leq_{K}\mathcal{I}|_{X}$ for some $X\notin \mathcal{I}$.
Moreover, either $\MM_{\mathcal{I}}$ has the constant or 1-1 property or adds Cohen reals.
\end{theorem}

One of the implications of the first part of the theorem works for Laver trees with an identical proof. We include it for the sake of completeness 
\begin{proposition}\label{LaverAddsCohen}
If $\nwd\leq_{K}\mathcal{I}|_{X}$ for some $X\notin \mathcal{I}$ then $\LL_{\mathcal{I}}$ adds Cohen reals.
\end{proposition}
\begin{proof}
    Suppose that $\pi:X\rightarrow 2^{<\omega}$ is such that $\pi^{-1}[Y]\in\mathcal{I}$ for every $\nwd$-positive $Y\subseteq 2^{<\omega}$. The real $\dot{c}$ given by $\dot{c}=\pi(\dot{r}_{gen}(0))^{\frown}\pi(\dot{r}_{gen}(1))^{\frown}...$ is a Cohen real.
\end{proof}

We were not able to prove the opposite implication for Laver trees (see Question \ref{doesLavAddCoh}).

\begin{proposition}
If $\MM_{\mathcal{I}}$ does not have the constant or 1-1 property, then $\LL_{\mathcal{I}}$ does not have it either.
\end{proposition}
\begin{proof}
Assume $\LL_{\mathcal{I}}$ has the constant or 1-1 property, $T\in\MM_{\mathcal{I}}$ and $f:[T]\rightarrow 2^{\omega}$ is continuous. Let $T'\in\LL_{\mathcal{I}}$ be the unique tree which is constructed from $T$ by "removing the parts in between the splinodes", i.e. $stem(T')=stem(T)$ and if for any tree $T$, if $s^{T}:\omega^{<\omega}\rightarrow split(T)$ is the unique $\subseteq$-isomorphism which preserves lexicographical order then $succ_{T'}(s^{T'}(stem^{\frown}\tau))=succ_{T}(s^{T}(stem^{\frown}\tau))$ for each $\tau\in\omega^{<\omega}$. Next let $f':[T']\rightarrow 2^{\omega}$ be a continuous function defined for $x\in[T']$ as $f'(x)=f(\bigcup_{k}s^{T}(\tau_{k}))$ where $\{\tau_{k}:k\in\omega\}\subseteq\omega^{<\omega}$ is such that $\{x_{n}: n\in\omega\}\cap split(T')=\{s^{T'}(\tau_{k}):k\in\omega\}$. Now clearly any restriction $S'\in\LL_{\mathcal{I}}$, $S'\leq T'$ with $f'|_{[S']}$ being constant or 1-1 produce $S\in\MM_{\mathcal{I}}$, $S\leq T$ with constant or 1-1 restriction $f|_{[S]}$.
\end{proof}

We will now focus on proving that $\LL_{\mathcal{I}}$ has the constant of 1-1 property for a certain class of not near maximal $P^{+}$-ideals. Recall that a filter $\mathcal{F}$ is \emph{near maximal} if there is a finite to one $h: \omega \rightarrow \omega$ such that $f(\mathcal{F})$ is an ultrafilter. An ideal $\mathcal{I}$ is \emph{near maximal} if the dual filter $\mathcal{I}^*$ is near maximal. By Talagrand's lemma (see \cite{talagrand} or \cite{BJ}), all meager ideals are not near maximal, so in particular all analytic ideals are not maximal. Note that, if $\ideal{I}$ is a $P^+$-ideal, and $h$ is a finite to one function such that $h(\ideal{I})$ is a maximal ideal, then $h(\ideal{I})^*$ has to be a $P$-point. Therefore, it is consistent that all $P^{+}$-ideals are not near maximal. On the other hand, under filter dichotomy, all non-meager ideals are near maximal (see \cite{BlassShelah}). The proof of the main theorem of this section resembles the proof of the analogous theorem for the usual Laver trees, which appears in \cite{BreKhWoh}. Before proving this, we need a list of several lemmas.

\begin{lemma}\label{star_one}
    If $\mathcal{I}$ is a $P^+$-ideal, then for every $T\in\LL_{\mathcal{I}}$ and every continuous $f:[T]\rightarrow 2^{\omega}$ there is $S\leq_{0}T$ and $z\in 2^{\omega}$ such that for every $N\in\omega$, $\{i\in succ_{S}(stem): f[[S|_{stem^{\frown}i}]]\nsubseteq [z|_{N}]\}\in\Fin.$
\end{lemma}
\begin{proof}
    Assume that $T\in\LL_{\mathcal{I}}$ and that $f:[T]\rightarrow 2^{\omega}$ is continuous. Using the pure decision property, we construct a sequence $\{T_{n}:n\in\omega\}\subseteq\LL_{\mathcal{I}}$ such that $T_{0}=T$, $T_{n+1}\leq_{0} T_{n}$ and that $T_{n}$ decides the value of $f(\dot{r}_{gen})|_{n}$ to be $x_{n}$. Then, as $\mathcal{I}$ is $P^+$, there is a pseudointersection $A\in\mathcal{I}^{+}$ of the family $\{succ_{T_{n}}(stem(T)):n\in\omega\}$. Let $x=\bigcup_{n}x_{n}\in2^{\omega}$ and let $S\in\LL_{\mathcal{I}}$ be such that $stem(S)=stem(T)$, $succ_{S}(stem(S))=A$ and for $n\in A$, $S|_{stem(S)^{\frown}n}=T_{k}|_{stem(S)^{\frown}n}$ where $k\in\omega$ is the last such that $n\in succ_{T_{k}}(stem(S))$.
\end{proof}

\begin{lemma}\label{laverdiscrete}
    Let $\mathcal{I}$ be a P$^+$-ideal and let $\{ A_n : n \in \omega\} \subseteq \mathcal{I}^+$, $z \in 2^\omega$ and $f_n : A_n \rightarrow 2^\omega$ be such that, for all $n,k \in \omega$, $f^{-1}_n(\langle z \restriction k \rangle \cap f[A_n]) \in \mathcal{I}^+$ and $z \notin f_n[A_n]$. Then there are $B_n \subseteq A_n$ such that $B_n \in \mathcal{I}^+$ and $\bigcup_{n \in \omega} f_n[B_n]$ is discrete.
\end{lemma}
\begin{proof}
    Using the P$^+$ property, it is possible to restrict to some $A_n' \subseteq A_n$ such that $f_n[A_n']$ converges to $z$, and therefore $f_n[A_n']$ is discrete. Let
$$
B_n = A_n' \setminus f_n^{-1}(\langle z \restriction n \rangle)
$$
then $B_n \subseteq A_n$, $B_n \in \mathcal{I}^+$ and, since $\bigcup_{n \in \omega} f_n[B_n] \setminus f_n^{-1}(\langle z \restriction n \rangle)$ is finite for every $k \in \omega$, then  $\bigcup_{n \in \omega} f_n[B_n]$ is discrete.
\end{proof}

\begin{lemma}
    Let $\mathcal{I}$ be a P$^+$-ideal such that it is not near maximal, $A, B \in \mathcal{I}^+$, $z \in 2^\omega$, $f: A \rightarrow 2^\omega$ and $g: B \rightarrow 2^\omega$ be such that $z \notin f[A] \cup g[B]$ and both $f[A]$ and $g[B]$ converge to $z$. Then, there are $A' \subseteq A$ and $B' \subseteq B$ such that $f[A'] \cap g[B'] = \emptyset$.
\end{lemma}
\begin{proof}
    Let $F: A \rightarrow \omega$ defined by $F(n) = k$ if and only if $f(n) \restriction k = z \restriction k$ but $f(n) \restriction k + 1 \neq z \restriction k + 1$. Define $G : B \rightarrow \omega$ in a similar way, using $g$ instead of $f$. Notice that both $F$ and $G$ are finite to one functions, so there must be $A", B" \subseteq \omega$ such that $A_0 = F^{-1}(A"), A_1 = F^{-1}(\omega \setminus A")$, $B_0 = G^{-1}(B"), B_1 = G^{-1}(\omega \setminus B")$ are all in $\mathcal{I}^+$. Then it is easy to show that, either $A_0 \cap B_0$ and $A_1 \cap B_1$ are in $\mathcal{I}^+$, or $A_1 \cap B_0$ and $A_0 \cap B_1$ are in $\mathcal{I}^+$.    
\end{proof}

\begin{lemma}
    Let $\mathcal{I}$ be a P$^+$-ideal such that it is not near maximal, $\{ B_n : n \in \omega\} \subseteq \mathcal{I}^+$, $z \in 2^\omega$ and $f_n : B_n \rightarrow 2^\omega$ such that $B_n \in \mathcal{I}^+$, $f_n[B_n]$ converges to $z$, then there are $C_n \subseteq B_n$ such that $C_n \in \mathcal{I}^+$, $\bigcup_{n \in \omega} f_n[C_n]$ is discrete, and for each $i,j \in \omega$, $i \neq j$ implies that $f_i[C_i] \cap f_j[C_j] = \emptyset.$\label{star_two}
\end{lemma}

\begin{proof}
    By Lemma \ref{laverdiscrete}, we may assume that $\bigcup_{n \in \omega} f_n[B_n]$ is discrete. Recursively, we are going to construct a family of sequences $\{ S^i_j : i \leq j, j \in \omega \}$ such that, for each $ i \leq j, j \in \omega$,
    \begin{enumerate}
        \item $S^i_j \subseteq B_j$ is $\mathcal{I}^+$, 
        \item $S^j_j \subseteq S^i_j$,
        \item for $m \geq i$, $f_m[S^i_m] \cap f_{i}[S^i_{i}]$ is finite.
    \end{enumerate}
    Assume we were able to carry out the construction of $\langle S^i_j : i \leq j, j \in \omega \rangle$ for all $i < k$, we are going to construct $\langle S^k_j : k \leq j, j \in \omega \rangle$: Let $S_k' = \bigcap_{i < k}S^i_k$ and recursively apply the previous remark to get two  families $\langle S'_\ell : k < \ell \rangle $ and $\langle S^k_\ell : k < \ell \rangle$ such that,
    \begin{itemize}
        \item $S'_\ell, S^k_\ell \in \mathcal{I}^+$,
        \item $S'_{\ell + 1} \subseteq S'_\ell \subseteq S_k'$,
        \item for all $i < k$, $S^k_\ell \subseteq S^i_\ell$,
        \item $f_k [S'_\ell] \cap f_\ell [S^k_\ell] = \emptyset$.
    \end{itemize}
    Finally let $S^k_k$ be a positive pseudo-intersection of the family $\langle S'_\ell : k < \ell \rangle $. It is routine to check that $\langle S^k_j : k \leq j, j \in \omega \rangle$ satisfies the properties we are looking for.

    Finally, let $C_n = S^n_n \setminus f_n^{-1}(\bigcup_{i < n}f_i [S^i_i])$. It is easy to see that the family $\langle C_n : n \in \omega \rangle$ satisfies the desired properties.
\end{proof}

We are now ready to prove the main theorem of this section.

\begin{theorem} \label{pplusconstant1}
If $\mathcal{I}$ is a P$^+$-ideal and $\mathcal{I}$ is not near maximal, then $\LL_{\mathcal{I}}$ has the constant or 1-1 property.
\end{theorem}
\begin{proof}
Assume that $T\in\LL_{\mathcal{I}}$ and $ f:[T]\rightarrow 2^{\omega}$ is continuous. Using a standard fusion argument and Lemma \ref{star_one}, we can get a tree $S\leq_{0}T$ and a collection $\{z_{\sigma}:\sigma\in split(S)\}\subseteq 2^{\omega}$ such that, for each $\sigma\in split(S)$ we have 
 \begin{equation} \label{lem9p1}
         \forall N\in\omega\text{  }    
         \{i\in succ_{S}(\sigma): f[[S|_{\sigma^{\frown}i}]]\nsubseteq [z_{\sigma}|_{N}]\}\in\Fin.
 \end{equation}
This in particular means that the sequence $\{z_{\sigma^{\frown}n}:n\in\omega\}$ converges to $z_{\sigma}$. To decide whether the restriction we are looking for is constant or 1-1 we will use a rank argument. Define $rank:split(S)\rightarrow \omega_{1}$ as:

$$ rank(\sigma) =
\begin{cases}
    0 & \text{ if } \{n\in\omega: z_{\sigma^{\frown}n}\neq z_{\sigma}\}\in\mathcal{I}^{+}, \\
    \alpha & \text{ if } \alpha \text{ is the smallest such that }\{n\in\omega: rank(\sigma^{\frown}n)<\alpha \}\in\mathcal{I}^{+}.
\end{cases}
$$
    
Now we consider two cases:

Case 1: For some $\sigma\in split(S)$ the rank is not defined. This implies that $\{n\in\omega: rank(\sigma^{\frown}n)$ is defined $\}\in \mathcal{I}$. Using this, together with a standard fusion argument, we trim $S$ to get a $R\leq_{0} S_\sigma$ such that for all $\tau\in split(R)$ we have $z_{\tau}=z_{\sigma}$. We conclude that the restriction $f|_{[R]}$ is constant with value $z_{\sigma}$.

Case 2: The rank is defined for all $\sigma\in split(S)$. We will inductively construct an increasing sequence $\{F_{n}:n\in\omega\}$ of well-founded subtrees of $S$ and disjoint clopen sets $\{V_{\sigma|_{|\sigma|-1}}:\sigma\in term(F_{n})\}$ such that for all $\sigma\in\omega^{<\omega}$ and $n\in\omega$ we have:
    \begin{enumerate}[label=\Alph*.]
        \item if $\sigma\in F_{n}$ is not a terminal node of $F_{n}$ then $succ_{F_{n}}(\sigma)\in\mathcal{I}^+$,
        \item if $\sigma\in F_{n+1}\setminus F_{n}$ then $\sigma$ extends a terminal node of $F_{n}$,
        \item for each $\sigma\in term(F_{n})$ we have $f[[S_{\sigma}]]\subseteq V_{\sigma|_{|\sigma|-1}}$.
    \end{enumerate}
 Assume for a moment that such a sequence has been constructed and let $R=\bigcup_{n}F_{n}$. Clearly $R\leq S$, and the condition A implies $R\in\LL_{\mathcal{I}}$. We claim that the restriction $f|_{[R]}$ is 1-1. Let $x,y\in[R]$, and $x\neq y$. Then there are $\sigma,\tau\in\omega^{<\omega}$, $\sigma\neq\tau$ and $n\in\omega$ such that $\sigma\subseteq x, \tau\subseteq y$ and $\sigma,\tau$ are terminal nodes of $F_{n}$. The condition C guarantees that $f(x)\in V_{\sigma|_{|\sigma|-1}}$ and $f(y)\in V_{\tau|_{|\tau|-1}}$ so $f(x)\neq f(y)$.
 
To finish the proof, we will show how to carry out the inductive construction: Let $F_{0}=\{stem(S)\}$ and assume that $F_{n}$ has been built. Fix a terminal node $\sigma$ of $F_{n}$. As the rank of $\sigma$ is defined, there is a well-founded tree $G_{\sigma}\subseteq\{\tau\in\omega^{<\omega}:\tau\subseteq\sigma$ or $\sigma\subseteq\tau\}$ such that all terminal nodes of $G_\sigma$ have rank $0$ and all non-terminal nodes of $G_\sigma$ splits into $\mathcal{I}$-positive sets. We then apply Lemma \ref{star_two} to the family $\{A_{\tau}:\tau\in term(G_{\sigma})\}\subseteq\mathcal{I}^{+}$, $A_{\tau}=\{n\in\omega: z_{\tau^{\frown}n}\neq z_{\tau}\}$, and the functions $f_{\tau}:A_{\tau}\rightarrow 2^{\omega}$ defined as $f_{\tau}(n)=z_{\tau^{\frown}n}$. We obtain $B_{\tau}\subseteq A_{\tau}$, $B_{\tau}\in\mathcal{I}^{+}$ such that $D=\bigcup\{f_{\tau}[B_{\tau}]:\tau\in term(G_{\sigma})\}$ is a discrete subset of $2^{\omega}$ and for $\tau \neq \tau'$, we have that $f_{\tau}[B_{\tau}] \cap f_{\tau'}[B_{\tau'}] =\emptyset$. Let then $\{U_{z}:z\in D\}$ be a collection of disjoint clopen sets such that $z\in U_{z}$ for each $z\in D$. By the condition (\ref{lem9p1}), we can remove $\mathcal{I}$-many points from each $B_{\tau}$ to obtain $\mathcal{I}$-positive sets $C_{\tau}\subseteq B_{\tau}$ such that for every $m\in C_{\tau}$ we have $f[[S|_{\tau^{\frown}m}]]\subseteq U_{z_{\tau^{\frown}m}}$. Let $H_{\sigma}$ be the tree generated by the set $\{\tau^{\frown}m:\tau\in term(G_{\sigma})$ and $m\in C_{\tau}\}$. Finally, we unfix $\sigma$ and let $F_{n+1}$ be the union of all $H_{\sigma}$'s where $\sigma$ is a terminal node of $F_{n}$. This finishes the construction.
\end{proof}

As a corollary we get the following result.

\begin{corollary}
    If $\mathcal{I}$ is a $P^+$-ideal that is not near maximal, then $\mathrm{cof}(\ell^{0}(\mathcal{I}))>\continuum$.
\end{corollary}

We do not know if this characterizes having the constant or 1-1 property for $\mathcal{I}$-Laver trees. In particular we do not know the answer of the following question.

\begin{question}
    Does $\LL_{\mathcal{Z}}$ or $\LL_{\FF}$ have the constant or 1-1 property?
\end{question}

\section{Incompatibility Shrinking Number}

Suppose that $\TT$ is a combinatorial tree forcing. The incompatibility shrinking number $\mathrm{is}(\TT)$ is the minimal amount of sets from $t^{0}_{\mathcal{B}}$ which intersects every $[T]$ for $T\in\TT$. Motivated by Theorem \ref{BreKhoWoh}, in this section we will estimate $\mathrm{is}(\TT)$ for the $\mathcal{I}$-Laver and $\mathcal{I}$-Miller trees in order to see in which cases the cofinality of the respective tree ideal is strictly greater than $\mathfrak{c}$. We mentioned the following simple estimates for the incompatibility shrinking number.

\begin{proposition}\label{ISBasicIneq}
$\mathrm{add}(t^{0}_{\mathcal{B}})\leq \mathrm{is}(\TT)\leq \mathrm{cov}(t^{0}_{\mathcal{B}})$. \label{addiscov}
\end{proposition}

This means that $\mathrm{add}(t^{0}_{\mathcal{B}})=\continuum$ implies $\mathrm{cof}(t^{0})>\continuum$. For $\FM$ the additivity is equal to $\omega_{1}$ (see \cite{RosNew}), so is for several other tree types as we will see in the last section of this manuscript (see Corollary \ref{AddEqualOmegaOne}). We will focus on improving these estimates. For this purpose, we will frequently use the following dichotomies. The first one is from Rosłanowski and Newelski.

\begin{theorem}[\cite{RosNew}]\label{RosNewDich}
    For every analytic $A\subseteq\omega^{\omega}$ either there is $T\in \FM$ such that $[T]\subseteq A$ or for some $\phi:\omega^{<\omega}\rightarrow\omega$, $A$ is contained in a set of the form 
    \begin{center}
        $F_{\phi}:=\{x\in \omega^{\omega}: \forall^{\infty} n\in\omega$ $x(n)\neq\phi(x|_{n})\}$.
    \end{center}
\end{theorem}

A similar dichotomy appears in a paper of Sabok and Zapletal.

\begin{theorem}[\cite{SabZapl}]\label{SabokZaplDich}
    For every analytic $A\subseteq\omega^{\omega}$ either there is $T\in \MM_{\mathcal{I}}$ such that $[T]\subseteq A$ or for some $\phi:\omega^{<\omega}\rightarrow \mathcal{I}$, $A$ is contained in a set of the form 
    \begin{center}
        $M_{\phi}:=\{x\in \omega^{\omega}: \forall^{\infty} n\in\omega$ $x(n)\in\phi(x|_{n})\}$.
    \end{center}
\end{theorem}

This implies that the Borel part $m^{0}_{\mathcal{B}}(\mathcal{I})$ is the $\sigma$-ideal generated by the sets $M_{\phi}$ for $\phi:\omega^{<\omega}\rightarrow \mathcal{I}$ and the forcing notion $\MM_{\mathcal{I}}$ is forcing equivalent $\Borel\setminus m^{0}_{\mathcal{B}}(\mathcal{I})$. A well known example is when $\mathcal{I}=\Fin$. In such case $\MM_{\mathcal{I}}$ is the Miller forcing and the dichotomy was already proven by Kechris in \cite{kechris}. Here, the Borel part consists of $\sigma$-compact sets and it is well known that $\mathrm{add}(m^{0}_{\mathcal{B}})=\mathrm{non}(m^{0}_{\mathcal{B}})=\bbb$ and $\mathrm{cov}(m^{0}_{\mathcal{B}})=\mathrm{cof}(m^{0}_{\mathcal{B}})=\ddd$.
A similar dichotomy for Laver trees was proven by Miller.

\begin{theorem}[\cite{Miller}]\label{MillerDich}
    For every analytic $A\subseteq\omega^{\omega}$ either there is $T\in\LL_{\mathcal{I}}$ such that $[T]\subseteq A$ or for some  $\phi:\omega^{<\omega}\rightarrow \mathcal{I}$, $A$ is contained in a set of form 
    \begin{center}
        $L_{\phi}:=\{x\in \omega^{\omega}: \exists^{\infty} n\in\omega$ $x(n)\in\phi(x|_{n})\}$.
    \end{center}
\end{theorem}

By the same line of reasoning, the Borel part $\ell^{0}_{\mathcal{B}}(\mathcal{I})$ is the $\sigma$-ideal generated by the sets $L_{\phi}$ for $\phi:\omega^{<\omega}\rightarrow \mathcal{I}$ and the forcing notion $\LL_{\mathcal{I}}$ is forcing equivalent $\Borel\setminus \ell^{0}_{\mathcal{B}}(\mathcal{I})$. Similarly, if $\mathcal{I}=\Fin$ then $\LL_{\mathcal{I}}$ is the Laver forcing and the Borel part $\ell^{0}_{\mathcal{B}}(\Fin)$ is the $\sigma$-ideal of not strongly dominating subsets of $\omega^{\omega}$, investigated in \cite{DecoRep}. In this paper, the authors prove that $\mathrm{add}(\ell^{0}_{\mathcal{B}}(\mathcal{I}))=\mathrm{cov}(\ell^{0}_{\mathcal{B}}(\mathcal{I}))=\bbb$ and $\ddd=\mathrm{non}(\ell^{0}_{\mathcal{B}}(\mathcal{I}))=\mathrm{cof}(\ell^{0}_{\mathcal{B}}(\mathcal{I}))$.
Before we prove the first theorem of this section, recall that, for an ideal $\mathcal{I}$ on $\omega$ ,the \emph{additivity of} $\mathcal{I}$ is defined as
$$
\mathrm{add}^*(\mathcal{I}) = min\{ |\mathcal{A}| : \mathcal{A} \subseteq \mathcal{I} \text{ and } \forall X \in \mathcal{I} \exists A \in \mathcal{A} (A \nsubseteq^* X) \}.
$$
This cardinal invariant has been extensively studied. The reader interested can consult \cite{Hrus}. We will also consider a slightly modified version of this cardinal invariant:
$$
\mathrm{add}_{\omega}^*(\mathcal{I}) = min\{ |\mathcal{A}| : \mathcal{A} \subseteq \mathcal{I} \text{ and } \forall \{X_{n}:n\in\omega\} \subseteq \mathcal{I} \exists A \in \mathcal{A} \forall n\in\omega (A \nsubseteq^* X_{n}) \}.
$$
Clearly $\mathrm{add}^*(\mathcal{I}) \leq \mathrm{add}_\omega^*(\mathcal{I})$. First, we notice that these two cardinals are the same whenever $\mathcal{I}$ is a $P$-ideal.

\begin{proposition}\label{AddOmegaEqualAdd}
If $\mathcal{I}$ is a P-ideal then $\mathrm{add}_{\omega}^*(\mathcal{I})=\mathrm{add}^*(\mathcal{I})$.
\end{proposition}
\begin{proof}
    To prove $\mathrm{add}_{\omega}^*(\mathcal{I})\leq\mathrm{add}^*(\mathcal{I})$ let $\{X_{\alpha}:\alpha<\kappa\}\subseteq \mathcal{I}$ with $\kappa<\mathrm{add}_{\omega}^*(\mathcal{I})$. There are then $\{Y_{n}: n\in\omega\}\subseteq\mathcal{I}$ such that each $X_{\alpha}$ is almost-contained in one of the $Y_{n}$'s. Then, any pseudo-union $Y\in\mathcal{I}$ of the $Y_{n}$'s is a pseudo-union of the $X_{\alpha}$'s. This finishes the proof.
\end{proof}

We refer the reader to the next section to see the relationship between $\mathrm{add}_{\omega}^*(\mathcal{I})$ and other cardinal invariants. Right now, we will only focus on the relationship with the incompatibility shrinking number. Our contribution begins with the following estimate. 

\begin{proposition}\label{addStarLessThenAdd}
For any ideal $\mathcal{I}$, both $\mathrm{add}(m^{0}_{\mathcal{B}}(\mathcal{I}))$ and $\mathrm{add}(\ell^{0}_{\mathcal{B}}(\mathcal{I}))$ are greater or equal than $min\{\bbb,\mathrm{add}^{*}_{\omega}(\mathcal{I})\}$.
\end{proposition}
\begin{proof}
Let $\kappa<min\{\bbb,\mathrm{add}^{*}_{\omega}(\mathcal{I})\}$, $T\in\LL_{\mathcal{I}}$ and suppose we are given a collection $\{L_{\phi_{\alpha}}:\alpha<\kappa\}$ such that $\phi_{\alpha}:\omega^{<\omega}\rightarrow\omega$. For each $\sigma\in\omega^{<\omega}$, as $\kappa<\mathrm{add}^{*}_{\omega}(\mathcal{I})$ there is a increasing sequence $\{C^{\sigma}_{n}:n\in\omega\}\subseteq\mathcal{I}$ such that for each $\alpha<\kappa$ there is $n_{\alpha}\in\omega$ with $\phi_{\alpha}(\sigma)\subseteq^{*}C^{\sigma}_{n_{\alpha}}$. For each $\alpha<\kappa$ let $g_{\alpha}:\omega^{<\omega}\rightarrow\omega$ be defined as $g_{\alpha}(\sigma)=n_{\alpha}$ and $f_{\alpha}:\omega^{<\omega}\rightarrow\omega$ be defined such that $\phi_{\alpha}(\sigma)\subseteq C^{\sigma}_{n_{\alpha}}\cup f_{\alpha}(\sigma)$. Now let $g:\omega^{<\omega}\rightarrow\omega$ and $f:\omega^{<\omega}\rightarrow\omega$ be such that $g$ dominates all $g_{\alpha}$'s and $f$ dominates all $f_{\alpha}$'s. Let $\phi:\omega^{<\omega}\rightarrow\mathcal{I}$ be defined as $\phi(\sigma)=C^{\sigma}_{g(\sigma)}\cup f(\sigma)$. We claim that $\bigcup_{\alpha}L_{\phi_{\alpha}}\subseteq L_{\phi}$. This will finish the proof. Let $x\in L_{\phi_{\alpha}}$. Then for infinitely many $n$'s we have $x(n)\in \phi_{\alpha}(x|_{n})\subseteq C^{x|_{n}}_{n_{\alpha}}\cup f_{\alpha}(x|_{n})\subseteq C^{x|_{n}}_{g(x|_{n})}\cup f(x|_{n})=\phi(x|_{n})$ as the sequence of the $C^{x|_{n}}_{k}$'s, $k\in\omega$ is increasing.
\end{proof}

In the case of Miller-like trees, we have the following estimate that uses the hereditary version of the covering number, which is a well-known cardinal invariant (for example, see \cite{Zapl1}):
$$
\mathrm{cov}_h(\ideal{I}) = \min\{ |\mathcal{A}| : \mathcal{A} \subseteq \mathcal{I} \text{ such that }\bigcup \mathcal{A} \text{ covers a Borel } \ideal{I} \text{ positive set} \}.
$$

\begin{theorem}\label{ISmainIneq}
For any ideal $\ideal{I}$, the following inequalities hold.
$$\min\{\bbb,\mathrm{cov}_h(m^{0}_{\mathcal{B}}(\mathcal{I}) )\}\leq \mathrm{is}_h(\MM_{\mathcal{I}})\leq \mathrm{cov}_h(m^{0}_{\mathcal{B}}(\mathcal{I})).$$ 
In particular, if $\mathbb{M}_\mathcal{I}$ is homogeneous, then the $h$'s can be dropped. The inequalities hold for $\FM$ as well.
\end{theorem}
\begin{proof}
We will focus only on $\MM_{\mathcal{I}}$ as the proof for $\FM$ is similar. By Theorem \ref{SabokZaplDich}, we only have to show that, for any $T \in \MM_{\mathcal{I}}$, the inequalities 
$$\min\{\bbb,\mathrm{cov}(m^{0}_{\mathcal{B}}(\mathcal{I}) |_{[T]} )\}\leq \mathrm{is}(\MM_{\mathcal{I}} | T )\leq \mathrm{cov}(m^{0}_{\mathcal{B}}(\mathcal{I})|_{[T]})$$ 
hold. Assume that $\kappa<min\{\bbb,\mathrm{cov}(m^{0}_{\mathcal{B}}(\mathcal{I}) |_{[T]})\}$ and a family $\{\phi_{\alpha}:\alpha<\kappa\}$ where $\phi_{\alpha}:\omega^{<\omega}\rightarrow\mathcal{I}$.
Let $M'_{\phi_{\alpha}}=\{x\in \omega^{<\omega}:\forall n\in\omega$ $x(n)\notin\phi_{\alpha}(x|_{n})\}$. Inductively we will construct a sequence $\{F_{n}:n\in\omega\}$ of well-founded subtrees of $T$ such that 
\begin{enumerate}[itemsep=0.3mm]
    \item $F_{n+1}$ is an end-extension of $F_{n}$ and $term(F_{n})\subseteq split(T)$,
    \item for every $\sigma\in term(F_{n})$ if $n\in succ_{T}(\sigma)$ then there is unique $\tau\in term(F_{n+1})$ such that $\sigma^{\frown}n \subseteq\tau$,
    \item for every $\sigma\in term(F_{n})$ there are only finitely many $\tau\in term(F_{n+1})$ such that $[\tau]\cap M'_{\phi_{\alpha}}=\emptyset$.
\end{enumerate}
Assume that $F_{n}$ has been constructed and fix $\sigma\in term(F_{n})$. For each $n\in succ_{T}(\sigma)$ we can choose a point $x^{\sigma}_{n}\in [\sigma^{\frown}n]\setminus\bigcup_{\alpha} M'_{\phi_{\alpha}}$ as $\kappa<\mathrm{cov}(m^{0}_{\mathcal{B}}(\mathcal{I}) |_{[T]})$ (and clearly this cardinal is the same when restricted to basic clopen sets). For every $\alpha<\kappa$ there is $f_{\alpha}\in\Baire$ such that 
\begin{center}
    $\forall n\in succ_{T}(\sigma)$ $x^{\sigma}_{n}|_{f_{\alpha}(n)}\notin M'_{\phi_{\alpha}}$
\end{center}
As $\kappa<\bbb$ there must be a $g\in\Baire$ such that it dominates all $f_{\alpha}$'s. For $n\in\omega\land \sigma\in term(F_{n})$ let $\rho^{n}_{\sigma}$ be the shortest splitnode of $T$ which extends $x^{\sigma}_{n}|_{g(n)}$. Let then $F_{n+1}$ be the tree generated by the set $\{\rho^{n}_{\sigma}:n\in\omega\land \sigma\in term(F_{n})\}$. Let $S=\bigcup_{n}F_{n}\in\MM_{\mathcal{I}}$. The condition (3) implies that $[S]\cap M'_{\phi_{\alpha}}$ is compact and therefore bounded by some $h_{\alpha}\in\baire$. Let $h\in\baire$ be such that it dominates all $h_{\alpha}$'s and let $R\in\MM_{\mathcal{I}}$ be such that $[R]$ consists of branches of $[S]$ which are not dominated by $h$. Then, clearly $[R]\cap\bigcup_{\alpha}M'_{\alpha}=\emptyset$ for all $\alpha<\kappa$.
\end{proof}

Since each $M_\phi$ is a meager set, then we have that $\mathrm{cov}(\mathcal{M}) \leq \mathrm{cov}(m^{0}_{\mathcal{B}}(\mathcal{I}) |_{[T]} )$, and therefore the previous theorem implies that $\min \{ \bbb, \mathrm{cov} (\mathcal{M}) \} \leq \mathrm{is}_h(\MM_\mathcal{I})$. For the types of trees that add Cohen reals, we are able to show that this lower bound is precise. For proving this, we will use the following result of J. Brendle from \cite{Brendle}.

\begin{theorem}\label{BreMillerOfCohen}
    If $M$ is submodel of $V$ and $V$ contains a Miller tree of Cohen real over $M$, then $V$ contains a dominating real over $M$.
\end{theorem}

We are ready to prove the following theorem.

\begin{theorem}
If $\MM_{\mathcal{I}}$ adds Cohen reals then $\mathrm{is}_h(\MM_{\mathcal{I}}) \leq min\{\bbb,\mathrm{cov}(m^{0}_{\mathcal{B}}(\mathcal{I}))\}$. \label{millercohenreals}
\end{theorem}
\begin{proof}
    Let $W \in \MM_\mathcal{I}$ and let $f : \baire \rightarrow \baire$ be continuous such that $W\forces{f(\dot{r}_\text{gen}) \text{ is Coh}$ $\text{en over }V}$. We will show that $\mathrm{is}(\MM_{\mathcal{I}} | W) \leq min\{\bbb,\mathrm{cov}(m^{0}_{\mathcal{B}}(\mathcal{I})|_{[W]})\}$. We will need the following observation.

\begin{claim}
    For every $N \subseteq \baire$ closed nowhere dense we have $f^{-1}[N] \in m^0_\mathcal{B}(\mathcal{I})|_{[W]}$.
\end{claim}
\begin{proof}[Proof of the claim.]
    Notice that $f^{-1}[N]$ is analytic, therefore by dichotomy for $\MM_{\mathcal{I}}$ trees, we have that either $f^{-1}[N]$ contains the branches of a tree $T \in \MM_{\mathcal{I}}$ or $f^{-1}[N] \in m^0_\mathcal{B}(\mathcal{I})$. The former is impossible as it would imply that $T\forces{f(\dot{r}_\text{gen}) \in N}$, so the conclusion follows.
\end{proof}
     In particular, for every $T \in \MM_{\mathcal{I}} | W$, $f[[T]]$ is not $\sigma$-compact and therefore, using Theorem \ref{SabokZaplDich} for Miller trees, $f[[T]]$ contains the branches of a tree $S_T \in \mathbb{M}$. Assume that $min\{\bbb,\mathrm{cov}(m^{0}_{\mathcal{B}}({\mathcal{I}})|_{[W]})\}<\mathrm{is}(\MM_{\mathcal{I}})$. As $\mathrm{is}(\MM_{\mathcal{I}} | W)\leq \mathrm{cov}(m^{0}_{\mathcal{B}}(\mathcal{I}) |_{[W]})$, then $\bbb<\mathrm{is}(\MM_{\mathcal{I}} | W)$. Pick $\mathcal{F} = \{f_{\alpha}:\alpha<\bbb\}\subseteq\omega^{\omega}$ an unbounded family and let $\mathbf{M}$ an elementary substructure of $H(\kappa)$ such that $|\mathbf{M}| = \bbb$ and $\mathcal{F} \cup \{F,f \} \subseteq \mathbf{M} $. As $\bbb<\mathrm{is}(\MM_{\mathcal{I}} | W)$ there is a tree $T\in\MM_{\mathcal{I}} | W$ such that $[T]\cap M_{\phi}=\emptyset$ for each $\phi:\omega^{<\omega}\rightarrow\mathcal{I}$, $\phi\in \mathbf{M}$. Then, by elementarity and by the previous remark, for every nowhere dense set $N \in \mathbf{M}$ there is $\phi \in \mathbf{M}$ such that $f^{-1}[N] \subseteq M_\phi$, which imply that every $x \in [S_T]$ is Cohen over $\mathbf{M}$, so by Theorem \ref{BreMillerOfCohen}, there must be a real $d \in \baire$ such that $d$ dominates every $y \in \baire \cap \mathbf{M}$, which is impossible since $\mathbf{M}$ contains an unbounded family.
\end{proof}

In \cite{RosNew}, it is proven that $\FM$ adds Cohen reals and that $\mathrm{cov}(fm^{0}_{\mathcal{B}})=\mathrm{cov}(\Meager)$. It is also well known that $\mathrm{add}(\mathcal{M})=min\{\bbb, \mathrm{cov}(\Meager)\}$ (for a proof, consult \cite{BJ}). As a corollary we get the following.

\begin{corollary}
$\mathrm{is}_h(\FM)=\mathrm{is}(\FM)=\mathrm{add}(\mathcal{M})$.
\end{corollary}

An immediate consequence is that, under $\mathrm{add}(\mathcal{M})=\continuum$ we have $\mathrm{cof}(fm_{0})>\continuum$, which was known only under CH (\cite{BreKhWoh}). One can easily show, using the same ideas, that both $\mathrm{is}(\MM_{\nwd})$ and $\mathrm{is}(\LL_{\nwd})$ are equal to $\mathrm{add}(\mathcal{M})$. A general answer to the main problem remains unknown. 

\begin{question}
    Is is consistent that $\mathrm{cof}(fm^{0})=\continuum$?
\end{question}

As we seen in Proposition \ref{addiscov}, the additivity and the covering of the Borel part of these ideals are heavily involved with the incompatibility shrinking number. In the next section, we will focus on estimating them.

\section{Cardinal invariants of the Borel parts}

This section is devoted to calculate some of the cardinal invariants of the Borel parts $m^{0}_{\mathcal{B}}(\mathcal{I})$ and $\ell^{0}_{\mathcal{B}}(\mathcal{I})$. We will focus only on the additivity and the covering numbers. Similar work has been done in \cite{BreSh}, where the authors calculate these in the case when $\mathcal{I}$ is a maximal ideal. We start with the following general fact concerning the trees that add Cohen reals.

\begin{theorem}\label{treeAddsCohen}
	Let $\TT$ be a combinatorial tree type with continuous reading of names which adds Cohen reals. Then $\mathrm{add}(t^{0})\leq \mathrm{add}(\Meager)$ and $\mathrm{cov}(t^{0})\leq \mathrm{cov}(\Meager)$. These inequalities hold also for the Borel part of $t^{0}$.
\end{theorem}
\begin{proof}
Let $\dot{c}$ be a $\TT$-name for a Cohen real and let $f:\omega^{\omega}\rightarrow\Cantor$ be a continuous function such that $\forces{f(\dot{r}_{gen})=\dot{c}}$. We define a function  $F:\Meager\rightarrow t^{0}$ such that for $M\in\Meager$ we have $F(M)=\{x\in\omega^{\omega}:f(x)\in M\}$. First, we will check that 
	\begin{claim}
		$F(M)\in t^{0}$.
	\end{claim}
 \begin{proof}[Proof of the Claim]
    Write $M$ as an union $\bigcup_{n<\omega}N_{n}$ of nowhere dense sets $N_{n}$. As $F(M)=\bigcup_{n}f^{-1}[N_{n}]$ it is enough to show that $f^{-1}[N]\in t^{0}$ for a nowhere dense $N$. Let $T\in\TT$ be arbitrary. Let $S\in\TT$, $S\leq T$ be such that $S\forces{f(\dot{r}_{gen})\notin N}$. As $N$ is closed and $f$ continuous there is $\sigma\in \omega^{<\omega}$ and $R\leq S$ such that $R\forces{ \sigma\subseteq\dot{r}_{gen} \text{ and } [\sigma]\cap f^{-1}[N]=\emptyset}$. By absolutness $[\sigma]\cap f^{-1}[N]=\emptyset$ and $[R]\subseteq[\sigma]$ which finishes the proof of the claim.
\end{proof}
	 Note at this point, that if for some family $\{M_{\alpha}:\alpha<\kappa\}$ of meager subsets of $\Cantor$ we have $\bigcup_{\alpha<\kappa}M_{\alpha}=\Cantor$ then clearly $\bigcup_{\alpha<\kappa}F(M_{\alpha})=\omega^{\omega}$ which justifies $\mathrm{cov}(t^{0})\leq \mathrm{cov}(\Meager)$. We will check that, for every family $\{M_{\alpha}:\alpha<\kappa\}\subseteq\Meager$,
\begin{center}
    if  $\bigcup_{\alpha<\kappa}F(M_{\alpha})\in t^{0}$ then $\bigcup_{\alpha<\kappa}M_{\alpha}\in\Meager$.
\end{center}
    This clearly implies that $\mathrm{add}(t^{0})\leq \mathrm{add}(\Meager)$ and will finish the proof of the theorem.
	Assume then that for some $\{M_{\alpha}:\alpha<\kappa\}\subseteq\Meager$ we have $\bigcup_{\alpha<\kappa}F(M_{\alpha})\in t^{0}$. Then there are maximal antichains $\mathcal{A}_{n}\subseteq \TT$, $n\in\omega$ such that $\bigcup_{\alpha<\kappa}F(M_{\alpha})\subseteq \bigcup_{n} (\omega^{\omega}\setminus\bigcup_{T\in\mathcal{A}_{n}}[T])$ which means that
	\begin{center}
 $(\bigcup_{\alpha<\kappa}F(M_{\alpha}))\cap(\bigcap_{n\in\omega}\bigcup_{T\in\mathcal{A}_{n}}[T])=\emptyset$.
	\end{center}
	Denote by $c_{T}$ the longest $\sigma\in\omega^{<\omega}$ such that $T\forces{c_{T}\subseteq \dot{c}}$. Without loss of generality we may assume that $\mathcal{A}_{n+1}$ refines $\mathcal{A}_{n}$ and that for every $n\in\omega$ and every $T\in\mathcal{A}_{n}$ we have $|stem(T)|,|c_{T}|>n$. Now we define the open sets $U_{n}=\bigcup\{[c_{T}]:T\in\mathcal{A}_{n}\}$. 
	\begin{claim}
		$U_{n}$'s are dense.
	\end{claim}
 \begin{proof}[Proof of the claim]
Take any $\sigma\in2^{<\omega}$. As the set $f^{-1}[[\sigma]]$ is open there must be $T\in\mathcal{A}_{n}$ and $S\leq T$ with $[S]\subseteq f^{-1}[[\sigma]]$. By absoluteness $S\forces{\dot{r}_{gen}\in f^{-1}[[\sigma]]}$ and then $S\forces{\sigma\subseteq f(\dot{r}_{gen})}$ and so $\sigma\subseteq c_{S}$ which finish the claim.
 \end{proof}
To finish the proof of the theorem we show that	$(\bigcup_{\alpha<\kappa}M_{\alpha})\cap(\bigcap_{n\in\omega}U_{n})=\emptyset$. Aiming for a  contradiction assume that there is a $x\in \Cantor$ such that $x\in M_{\alpha}$ and $x\in\bigcap_{n\in\omega}U_{n}$. We want to construct $y\in F(M_{\alpha})$ with $y\in \bigcap_{n\in\omega}\bigcup_{T\in\mathcal{A}_{n}}[T]$ which will contradict $(\bigcup_{\alpha<\kappa}F(M_{\alpha}))\cap(\bigcap_{n\in\omega}\bigcup_{T\in\mathcal{A}_{n}}[T])=\emptyset$. For every $n\in\omega$ as $x\in U_{n}$ we have that there is $T_{n}\in\mathcal{A}_{n}$ with $c_{T_{n}}\subseteq x$. This implies that $\bigcup_{n\in\omega}c_{T_{n}}=x$. Let $y=\bigcup_{n}stem(T_{n})$. Then, observe that $f(y)=x\in M_{\alpha}$ so $y\in f^{-1}[M_{\alpha}]=F(M_{\alpha})$. As $\mathcal{A}_{n+1}$ refines $\mathcal{A}_{n}$ we have $y\in[T_{n}]$ for each $n\in\omega$. 
\end{proof}

According to Theorem \ref{SabZapl} and Proposition \ref{LaverAddsCohen}, the previous estimates apply to the tree types $\MM_{\nwd}$, $\LL_{\nwd}$ as well as $\FM$ (\cite{RosNew}). We will now focus on calculating the additivities.

\begin{theorem}\label{UpperBoundAdd}
For any ideal $\mathcal{I}$, both $\mathrm{add}(\ell^{0}_{\mathcal{B}}(\mathcal{I}))$ and $\mathrm{add}(m^{0}_{\mathcal{B}}(\mathcal{I}))$ are smaller or equal to $\mathrm{add}^*_{\omega}(\mathcal{I})$ .
\end{theorem}
\begin{proof}
Let $\{B_\alpha : \alpha \in \kappa\}$ be a witness for $\mathrm{add}^*_{\omega}(\mathcal{I})$. For each $\alpha \in \kappa$ define $\varphi_\alpha: \omega^{<\omega} \rightarrow \mathcal{I}$ as $\varphi_\alpha (s) = B_\alpha$. We will see that $\bigcup_{\alpha \in \kappa} L_{\varphi_\alpha} \notin \ell^{0}_{\mathcal{B}}(\mathcal{I})$.

Aiming for a contradiction, assume that there is a sequence $\{ \phi_{n}: n\in\omega\}$ such that, for all $n\in\omega, \phi_{n} : \omega^{<\omega} \rightarrow \mathcal{I}$ and $\bigcup_{\alpha \in \kappa} L_{\varphi_\alpha} \subseteq \bigcup_{i \in \omega} L_{\phi_i}$. Then find $\alpha \in \kappa$ such that for all $i \in \omega$ and all $s \in \omega^{<\omega}$, $B_\alpha \setminus \bigcup_{\langle i,s\rangle \in F} \phi_i(s) \neq \emptyset$ for every finite $F \subseteq \omega \times \omega^{<\omega}$. We will show that $L_{\varphi_\alpha} \setminus \bigcup_{i \in \omega} L_{\phi_i} \neq \emptyset$:
Recursively construct $x:\omega \rightarrow \omega$ such that, for all $n \in \omega$
\begin{itemize}
    \item $x(n) \in B_\alpha$,
    \item $x(n) \notin \phi_i(x | k)$ for every $i,k \leq n$.
\end{itemize}
Since $x(n) \in \varphi_\alpha (x | n)$ for all $n\in \omega$, then $x \in L_{\varphi_\alpha}$. On the other hand, for all $i\in \omega, x \notin \phi_i (x | n)$ for all $n > i$, therefore $x \notin L_{\phi_i}$.
\end{proof}

By the previous Theorem and by Proposition \ref{addStarLessThenAdd}, we know that $\min\{\bbb,\mathrm{add}^{*}_{\omega}(\mathcal{I})\}\leq\mathrm{add}(t^{0}_{\mathcal{B}}(\mathcal{I}))\leq\mathrm{add}^*_{\omega}(\mathcal{I})$ hold for both $\ell^{0}_{\mathcal{B}}(\mathcal{I})$ and $m^{0}_{\mathcal{B}}(\mathcal{I})$. Thus we will now calculate the cardinal invariants $\mathrm{add}^{*}_{\omega}$ for some of the Borel ideals, starting with $\FF$.

\begin{proposition}\label{FFAddStarOmega}
    $\mathrm{add}_{\omega}^*(\FF)=\bbb$.
\end{proposition}
\begin{proof}
    Notice that an unbounded family of functions witnesses $\mathrm{add}_{\omega}^*(\FF)$ and less than $\bbb$ sets from $\FF$ cannot, therefore $\mathrm{add}_{\omega}^*(\FF)=\bbb$.
\end{proof}

\begin{proposition}
    $\mathrm{add}_{\omega}^{*}(\ED)=\mathrm{add}_{\omega}^{*}(\ED_{fin})=\omega_{1}$.
\end{proposition}
\begin{proof}
We will focus only on $\ED$ as the proof for $\ED_{fin}$ is identical. Let $[N,F]=:\{(n,m):n<N $ or $m\in\{f(n):f\in F\}\}$ for $N\in\omega$ and $F\in[\baire]^{N}$. Clearly any $A\in \ED$ can be covered by the set of this form. Fix an eventually different family $\{e_{\alpha}:\alpha<\omega_{1}\}\subseteq\omega^{\omega}$. We claim that $\{e_{\alpha}:\alpha<\omega_{1}\}$ is a witness for $add_{\omega}^{*}(\ED)$. To show this, assume that we are given $\{A_{n}:n\in\omega\}\subseteq\ED$ where $A_{n}=[N_{n},F_{n}]\in\ED$ where $N_{n}\in\omega$, $F_{n}\in[\baire]^{<\omega}$. Observe that for each
$f\in\omega^{\omega}$ the set $\{\alpha\in\omega_{1}: e_{\alpha}\text{ is almost contained in }\bigcup F_{n}\}$ is of cardinality at most $|F_{n}|$. Choose $\alpha_{0}\in\omega_{1}\setminus\bigcup_{n}\{\alpha: e_{\alpha}\text{ is almost contained in }\bigcup F_{n}\}$. Clearly $e_{\alpha_{0}}$ is as desired.
\end{proof}

\begin{proposition}
    $\mathrm{add}_{\omega}^{*}(\Solecki)=\omega_{1}$.
\end{proposition}
\begin{proof}
Let $\{x_{\alpha}:\alpha<\omega_{1}\}\subseteq 2^{\omega}$. Define $[F]=\{C\in\Omega: C\cap F\neq\emptyset\}$ where $F\in[2^{\omega}]^{<\omega}$. We claim that the collection $\{[\{x_{\alpha}\}]:\alpha<\omega_{1}\}$ is a  witness for $\mathrm{add}_{\omega}^{*}(\Solecki)$. To show this assume that we are given $\{A_{n}:n\in\omega\}\subseteq\Solecki$ where $A_{n}=[F_{n}]$ and $F_{n}\in[\Cantor]^{<\omega}$. Enumerate $\bigcup_{n}F_{n}$ as $\{y_{n}:n\in\omega\}$ and pick $x_{\alpha}\notin\{y_{n}:n\in\omega\}$. Clearly, for each $n\in\omega$, there are infinitely many clopen sets $C$ of measure $1/2$ which contain $x_{\alpha}$ but not $y_{n}$. This finishes the proof.
\end{proof}

In the case of the ideals $conv$ and $\RandomGraph$ we are able to get a stronger result. For that purpose, we introduce the $\omega$-version of $\mathrm{non}^*(\mathcal{I})$:
$$
\mathrm{non}_{\omega}^*(\mathcal{I}) = \min\{ |\mathcal{A}| : \mathcal{A} \subseteq [\omega]^{\omega} \text{ and } \forall \{X_{n}:n\in\omega\} \subseteq \mathcal{I} \exists A \in \mathcal{A} \forall n\in\omega(A \cap X_{n}=^*\emptyset) \}.
$$

Clearly, $\mathrm{add}_{\omega}^*(\mathcal{I})\leq \mathrm{non}_{\omega}^*(\mathcal{I})$. The reader can easily verify the following proposition.

\begin{proposition}\label{NonKatetov}
    The Katetov reducibility $\mathcal{I}_{0}\leq_{K}\mathcal{I}_{1}$ implies $\mathrm{non}_{\omega}^*(\mathcal{I}_{0})\leq \mathrm{non}_{\omega}^*(\mathcal{I}_{1})$.
\end{proposition}
We have the following
\begin{proposition}
    $\mathrm{non}_{\omega}^{*}(conv)=\mathrm{non}_{\omega}^{*}(\RandomGraph)=\omega_{1}$.
\end{proposition}
\begin{proof}
One can show that $\RandomGraph\leq_{K} conv$ (see \cite{HrushKatetov}), so by Proposition \ref{NonKatetov} it is enough to consider $\mathrm{non}_{\omega}^{*}(conv)$. Let $\{x_{\alpha}:\alpha<\omega_{1}\}\subseteq 2^{\omega}$, for each $\alpha<\omega_{1}$ choose $A_{\alpha}\subseteq \mathbb{Q}$ to be a sequence converging to $x_{\alpha}$. We claim that $\{A_{\alpha}:\alpha\in\omega_{1}\}\subseteq conv$  witnesses $\mathrm{non}_{\omega}^{*}(conv)$. If $\{B_{n}:n\in\omega\}\subseteq conv$ is given, choose $\alpha\in\omega_{1}$ such that $x_{\alpha}$ is not a convergence point of any $B_{n}$. Clearly then $A_{\alpha}\cap B_{n}$ is finite.
\end{proof}

As a corollary of the previous propositions we get the following.

\begin{corollary}\label{AddEqualOmegaOne}
    For the ideals $\ED, \ED_{fin}, \Solecki, conv$ and $\RandomGraph$ the additivities of the Borel parts $m^{0}_{\mathcal{B}}(\mathcal{I})$ and $\ell^{0}_{\mathcal{B}}(\mathcal{I})$ are equal to $\omega_{1}$.
\end{corollary}

Moving towards $\nwd$, recall that $\mmm_{\sigma-\text{centered}}$ is the Martin's number of $\sigma$-centered partial orderings: $\mmm_{\sigma-\text{centered}}$ is the smallest cardinal number $\kappa$ such that there is a $\sigma-\text{centered}$ partial order $\mathbb{P}$ and a collection of $\kappa$ many open dense subsets of $\mathbb{P}$ such that no filter in $\mathbb{P}$ intersects them all. 

\begin{theorem}\label{nwdAddStarOmega}
    $\mmm_{\sigma-\text{centered}}\leq\mathrm{add}_{\omega}^*(\nwd)\leq\mathrm{add}(\Meager)$.
\end{theorem}
\begin{proof}
To prove the second inequality simply notice that if $\{N_{\alpha}:\alpha<\kappa\}$ is a collection of nowhere dense subsets of $2^{\omega}$ such that $\bigcup_{\alpha}N_{\alpha}$ is not meager then the nowhere dense trees $\{T_{\alpha}:\alpha<\kappa\}\subseteq\nwd$, $[T_{\alpha}]=N_{\alpha}$ witness for $\mathrm{add}_{\omega}^*(\nwd)$.

To prove the first inequality, assume that $\kappa<\mmm_{\sigma-centered}$ and that we are given $\{T_{\alpha}:\alpha<\kappa\}$ with $T_{\alpha}\in \nwd$. We define a partial order $\mathbb{P}$ consisting of finite partial functions from $\omega$ to $\nwd\times\omega$. Let $p=\{(A^{p}_{k},i^{p}_{k}):k<n\}$. The order is given by $p\leq q$ if $dom(q)\subseteq dom(p)$ and for each $k\in dom(q)$ we have $i^{p}_{k}>i^{q}_{k}$, $A^{q}_{k}\subseteq A^{p}_{k}$ and $A^{p}_{k}\cap 2^{i^{q}_{k}}=A^{p}_{k}\cap 2^{i^{q}_{k}}$. First, observe that $\mathbb{P}$ is $\sigma$-centered: if $p$ and $q$ have the same domains, and for each $k\in dom(p)$, $i^{p}_{k}=i^{q}_{k}=i$ and $A^{p}_{k}\cap 2^{i}=A^{q}_{k}\cap 2^{i}$, then $p$ and $q$ are compatible. Second, note that for any $\alpha<\kappa$ the set $D_{\alpha}=\{ p\in\mathbb{P}: A^{p}_{k}=T_{\alpha} \text{ for some } k\in dom(p)\}$ is dense in $\mathbb{P}$.
If $G\subseteq\mathbb{P}$ is generic filter over the family $\{D_{\alpha}:\alpha<\kappa\}$ then, the sequence $\{B_{n}:n\in\omega\}\subseteq nwd$ defined as $B_{n}=\bigcup\{A^{p}_{n}\cap i^{p}_{n}: p\in G\}$ is such that for each $\alpha$ there is $n\in\omega$ with $T_{\alpha}\subseteq B_{n}$ which finishes the proof.
\end{proof}

In this last theorem, we were do not know if any of these inequalities can be strict, leaving the following as a question.

\begin{question}
    Is $\mathrm{add}_{\omega}^*(\nwd)=\mathrm{add}(\Meager)$ true?
\end{question}

One can show that $\mathrm{non}_{\omega}^{*}(\ED_{fin})=\mathrm{cov}(\Meager)$ (the ideas are similar to the proof of $\mathrm{non}^{*}(\ED_{fin})=\mathrm{cov}(\Meager)$ in \cite{Hrus}). However, the following is unknown to us.

\begin{question}
    What are the values of $\mathrm{non}_{\omega}^{*}(\Solecki)$ and $\mathrm{non}_{\omega}^{*}(\ED)$?
\end{question}

In the case of analytic $P$-ideals, it is well known that $\mathrm{add}(\Null)\leq\mathrm{add}^*(\mathcal{I})\leq\bbb$ (see \cite{Stevo} and \cite{Hernand}). This, together with Propositions \ref{addStarLessThenAdd} and \ref{GeneralIneq} proves the following.

\begin{corollary}\label{analyPidealsAdd}
For analytic P-ideals $\mathcal{I}$, $\mathrm{add}(\Null)\leq\mathrm{add}(t^{0}_{\mathcal{B}}(\mathcal{I}))\leq \bbb$ for both $\mathcal{I}$-Laver and $\mathcal{I}$-Miller trees.
\end{corollary}

For the ideals $\FF$ and $\nwd$ we have the following application.

\begin{corollary}\label{FFnwdAdd}
    $\mathrm{add}(t^{0}_{\mathcal{B}}(\FF))=\bbb$ and $\mmm_{\sigma-\text{centered}}\leq \mathrm{add}(t^{0}_{\mathcal{B}}(\nwd))\leq \mathrm{add}(\Meager)$ hold for both $\mathcal{I}$-Laver and $\mathcal{I}$-Miller trees.
\end{corollary}
\begin{proof}
    It follows from Theorems \ref{nwdAddStarOmega} and \ref{treeAddsCohen} together with Propositions \ref{addStarLessThenAdd} and \ref{FFAddStarOmega}.
\end{proof}

We will now focus on calculating the covering numbers. The $\mathcal{I}$-Miller trees always add an unbounded real and $\mathcal{I}$-Laver add a dominating real, regardless of the ideal $\mathcal{I}$. As a consequence we have.

\begin{proposition}\label{GeneralIneq}
    For any ideal $\mathcal{I}$, the inequalities $\mathrm{cov}(\Meager)\leq \mathrm{cov}(m^{0}_{\mathcal{B}}(\mathcal{I}))\leq \ddd$ and $\mathrm{cov}(\ell^{0}_{\mathcal{B}}(\mathcal{I}))\leq \bbb$ hold.
\end{proposition}
\begin{proof}
The first inequality follows from the fact that $M_{\phi}$'s are meager. For the second inequality let $\{f_{\alpha}:\alpha<\ddd\}\subseteq\omega^{\omega}$ be a dominating family. If $\phi_{\alpha}$ is such that $\phi_{\alpha}(\sigma)=f_{\alpha}(|\sigma|)$ then the sets $M_{\phi_{\alpha}}$ cover the entire $\omega^{\omega}$. For the last inequality pick an unbounded family instead of a dominating one.
\end{proof}

Note that Theorem \ref{treeAddsCohen} and Proposition \ref{GeneralIneq} imply that $\mathrm{cov}(m^{0}_{\mathcal{B}}(\mathcal{I}))=\mathrm{cov}(\Meager)$ for any ideal $\mathcal{I}$ such that $\MM_{\mathcal{I}}$ adds Cohen reals. In the case of Laver trees we have the following result.

\begin{proposition}\label{covLaverNWD}
    $\mathrm{cov}(\ell^{0}_{\mathcal{B}}(\nwd))=\mathrm{add}(\Meager)$.
\end{proposition}
\begin{proof}
It follows from Proposition \ref{LaverAddsCohen}, Theorem \ref{treeAddsCohen} and Proposition \ref{GeneralIneq} that $\mathrm{cov}(\ell^{0}_{\mathcal{B}}(\nwd))\leq \mathrm{add}(\Meager)$. To prove $\mathrm{add}(\Meager)\leq \mathrm{cov}(\ell^{0}_{\mathcal{B}}(\nwd))$ assume that we are given a family ${\phi_{\alpha}:\alpha<\kappa}$, $\phi_{\alpha}:\omega^{<\omega}\rightarrow \nwd$ where $\kappa<\mathrm{add}(\Meager)$. Without loss of generality we may assume that  $\phi_{\alpha}(\sigma)$ is a nowhere dense subtree of $2^{<\omega}$. As $\kappa<\mathrm{cov}(\Meager)$, we pick $z\in 2^{\omega}\setminus \bigcup_{\alpha, \sigma}[\phi_{\alpha}(\sigma)]$. For each $\alpha<\kappa$ define $f_{\alpha}:\omega^{<\omega}\rightarrow\omega$ such that $f_{\alpha}(\sigma)$ is any natural number $m\in\omega$ such that $x|_{m}\notin \phi_{\alpha}(\sigma)$. As $\kappa<\bbb$ there is $g:\omega^{<\omega}\rightarrow\omega$ which dominates all $f_{\alpha}$'s. Construct $x\in\omega^{\omega}$ inductively in such a way that $x(n)=z|_{g(x|_{n})}$. It follows that $x\notin\bigcup_{\alpha}L_{\phi_{\alpha}}$.
\end{proof}

There is a simple relationship between the covering numbers and Katetov reducibility.

\begin{proposition}
    The Katetov reducibility $\mathcal{I}_{0}\leq_{K}\mathcal{I}_{1}$ implies $\mathrm{cov}(m^{0}_{\mathcal{B}}(\mathcal{I}_{1}))\leq \mathrm{cov}(m^{0}_{\mathcal{B}}(\mathcal{I}_{0}))$ and $\mathrm{cov}(\ell^{0}_{\mathcal{B}}(\mathcal{I}_{1}))\leq \mathrm{cov}(\ell^{0}_{\mathcal{B}}(\mathcal{I}_{0}))$.  
\end{proposition}
\begin{proof}
Let $\pi:dom(\mathcal{I}_{1})\rightarrow dom(\mathcal{I}_{0})$ be Katetov reduction and let $\kappa<\mathrm{cov}(\ell^{0}_{\mathcal{B}}(\mathcal{I}_{1}))$. Given a family $\{\phi_{\alpha}:\alpha\in\kappa\}$ such that $\phi_{\alpha}:dom(\mathcal{I}_{0})^{<\omega}\rightarrow \mathcal{I}_{0}$, define $\psi_{\alpha}:dom(\mathcal{I}_{1})^{<\omega}\rightarrow \mathcal{I}_{1}$ as $\psi_{\alpha}(\sigma)=\pi^{-1}[\phi_{\alpha}(\pi(\sigma))]\in \mathcal{I}_{1}$. Now, if $x\notin\bigcup_{\alpha}L_{\psi_{\alpha}}$ then $\pi\circ x\notin\bigcup_{\alpha}L_{\phi_{\alpha}}$ which finish the proof. 
\end{proof}

As a corollary we get 

\begin{corollary}\label{covNumbEqBBB}
    For the ideals $\FF$, $\RandomGraph$, $\ED$ and $conv$, the covering numbers of the $\sigma$-ideal $\ell^{0}_{\mathcal{B}}(\mathcal{I})$ are equal to $\bbb$.
\end{corollary}

It follows implicitly that, for the ideals mentioned in the previous result, the $\mathcal{I}$-Laver trees do not add a Cohen real. Otherwise, by Theorem \ref{treeAddsCohen}, these covering numbers would be below $\mathrm{cov}(\Meager)$. We do not know if this is true in general for $\mathcal{I}$-Laver trees such that  $\mathcal{I}$ is not Katetov above of $nwd$. In particular we have the following question.

\begin{question}\label{doesLavAddCoh}
    Do one of the forcings $\LL_{\Solecki}$, $\LL_{\SUM}$ $\LL_{\mathcal{Z}}$ or $\LL_{\ED_{fin}}$ add a Cohen real?
\end{question}

The situation is not clear with the covering numbers of the $\mathcal{I}$-Miller trees. We know from the previous results that $\max\{\mathrm{cov}(\mathcal{M}),\bbb\} \leq  \mathrm{cov}(m^{0}_{\mathcal{B}}(\FF))\leq \ddd$ but we do not know any other inequalities.

\begin{question}
    Is it possible to consistently distinguish the covering numbers $\mathrm{cov}(m^{0}_{\mathcal{B}}(\mathcal{I}))$ for different $\mathcal{I}$'s? In particular, is $\mathrm{cov}(m^{0}_{\mathcal{B}}(\FF))<\ddd$ consistent? 
\end{question}

For our last proposition, $\mathcal{E}$ will denote the $\sigma$-ideal generated by closed sets of measure zero on $2^{\omega}$. Clearly $\mathcal{E}\subseteq \Meager\cap\Null$, the inclusion is strict (see \cite{BJ}).

\begin{proposition}\label{SoleckiCovering}
    $min\{\bbb,\mathrm{non}(\mathcal{E})\}\leq \mathrm{cov}(\ell^{0}_{\mathcal{B}}(\mathcal{S})), \mathrm{cov}(m^{0}_{\mathcal{B}}(\mathcal{S}))$.
\end{proposition}
\begin{proof}
Let $\kappa<\bbb$ and let $\{\phi_{\alpha}:dom(\mathcal{S})^{<\omega}\rightarrow \mathcal{S}:\alpha<\kappa\}$. As $\kappa<\mathrm{non}(\mathcal{E})$ the set $\bigcup_{\alpha<\kappa}\bigcup_{\sigma}\phi_{\alpha}(\sigma)$ can be  covered by countable increasing union $\bigcup_{n}N_{n}$ of closed sets of measure zero. Now for each $\alpha<\kappa$, define a function $f_{\alpha}:\omega^{<\omega}\rightarrow\omega$ as $f_{\alpha}(\sigma)=min\{m\in\omega: \phi_{\alpha}(\sigma)\subset N_{m}\}$. The function $f_{\alpha}$ is well defined as each set $\phi_{\alpha}(\sigma)$ is finite and the $N_{n}$'s are increasing. Let $g:\omega^{<\omega}\rightarrow\omega$ be a function dominating all the $f_{\alpha}$'s. We define $x\in dom(\mathcal{S})^{\omega}$ such that $x(n)$ is disjoint from $N_{g(x|_{n})}$. Then clearly $x\notin\bigcup_{\alpha<\kappa}L_{\phi_{\alpha}}$.
\end{proof}

We finish this work by summarizing the results of this section and relating them to the incompatibility shrinking number.

\begin{theorem}
\begin{enumerate}[itemsep=0.3mm]
    \item If $\mathcal{I}$ is either $conv$, $\RandomGraph$ or $\ED$, then we have that $\omega_{1}\leq\mathrm{is}(\LL_{\mathcal{I}})\leq\bbb$ and $\bbb\leq\mathrm{is}(\MM_{\mathcal{I}})\leq\ddd$,
    
    \item  $\mathrm{add}(\Null)\leq\mathrm{is}(\LL_{\mathcal{I}})\leq\bbb$ and $\mathrm{add}(\Null)\leq\mathrm{is}(\MM_{\mathcal{I}})\leq\ddd$ hold for any analytic $P$-ideals $\mathcal{I}$,

    \item $\mathrm{is}(\MM_{{\nwd}})=\mathrm{add}(\Meager)$ and $\mmm_{\sigma\text{-centered}}\leq \mathrm{is}(\LL_{{\nwd}}) \leq \mathrm{add}(\Meager)$,

    \item For $\mathcal{I} = \FF$ we have $\mathrm{is}(\LL_{\mathcal{I}})=\bbb$ and $\bbb\leq\mathrm{is}(\MM_{\mathcal{I}})\leq\ddd$,
    
    \item $\omega_{1}\leq\mathrm{is}(\LL_{\mathcal{S}})\leq\bbb$, $min\{\bbb,\mathrm{non}(\mathcal{E})\}\leq\mathrm{is}(\MM_{\mathcal{S}})\leq\ddd$.
\end{enumerate}
\end{theorem}
\begin{proof}
    For (1) follows from Proposition \ref{ISBasicIneq} and Corollaries \ref{AddEqualOmegaOne} and \ref{covNumbEqBBB}. For (2), it follows from \ref{ISBasicIneq} with \ref{analyPidealsAdd} and \ref{GeneralIneq}. (3): For the Miller case follows from \ref{SabZapl}, \ref{ISmainIneq}, \ref{millercohenreals} and the fact that $\nwd$ is weakly homogeneous. For the Laver case follows from \ref{FFnwdAdd} and \ref{covLaverNWD}. (4) follows from \ref{ISBasicIneq}, \ref{FFnwdAdd} and \ref{GeneralIneq}. For (5) use \ref{ISBasicIneq} with \ref{ISmainIneq}, \ref{AddEqualOmegaOne}, \ref{GeneralIneq} and \ref{SoleckiCovering}.
\end{proof}

We do not know any other relationships between these cardinal invariants. In particular we do not know the answer to the following question.

\begin{question}
    What are the values of $\mathrm{is}(\MM_{\FF})$, $\mathrm{is}(\LL_{\nwd})$,  $\mathrm{is}(\LL_{\Solecki})$ and $\mathrm{is}(\LL_{\ED_{fin}})$?
\end{question}

\begin{acknowledgement}
    The authors would like to thank the participants of the joint topology and set theory seminar from Uniwersytet Wrocławski and from the Wrocław University of Science and Technology for many hours of stimulating conversations. We also would like to thank Osvaldo Guzmán for pointing out a mistake in the proof of Theorem \ref{pplusconstant1}.
\end{acknowledgement}

\bibliographystyle{alphadin} 
\bibliography{main}

\end{document}